\documentclass[a4paper,12pt,reqno]{amsart}
\usepackage{dsfont}
\usepackage{amsmath,amsfonts,amssymb,amsthm,enumerate,multicol}
\usepackage{tikz,graphicx}
\usepackage{hyperref}
\usepackage{caption,enumitem,wasysym}
\usepackage{cleveref}
\usepackage{epstopdf}
\usepackage{float,wrapfig}
\usepackage[labelsep=space]{caption}
\usepackage[font=footnotesize]{caption}
\usepackage{xcolor}

\newtheorem{proposition}{Proposition}[section]
\newtheorem{definition}{Definition}[section]
\newtheorem{theorem}{Theorem}[section]
\newtheorem{corollary}{Corollary}[section]

\newtheorem{lemma}{Lemma}[section]

\newtheorem{remark}{Remark}[section]

\allowdisplaybreaks

 \theoremstyle{plain}

\usepackage{mathtools}

\usepackage{amsmath,amsfonts,amssymb,amsthm,enumerate,multicol}
\usepackage{tikz}
\usepackage{float}
\usepackage{pgfplots}
\pgfplotsset{compat=1.14}
\usepackage{tkz-graph}

\usepackage{autobreak}

\allowdisplaybreaks

\numberwithin{equation}{section}

\begin{document}

\begin{center}
\Large{\textbf{Global Existence and Mass Decay Analysis of solutions to the discrete Redner-Ben-Avraham-Kahng coagulation model}}
\end{center}
\medskip
\medskip
\centerline{${\text{Pratibha~Verma}}$ }
\centerline{${\text{pverma@ma.iitr.ac.in}}$}
\medskip
{\footnotesize

  \centerline{ ${}^{1}$Department of Mathematics, Indian Institute of Technology Roorkee,}
   \centerline{Roorkee-247667, Uttarakhand, India}

}

\bigskip

\section*{Abstract}
The Redner-Ben-Avraham-Kahng (RBK) coagulation model provide a fundamental framework for modeling the aggregation of particles in various physical and biological systems. In this paper, we investigate the global existence of solutions to the discrete version of RBK coagulation equations, encompassing a wide range of coagulation kernels. Furthermore, we demonstrate that the mass decays with time and will eventually reach to zero even if the original mass is not finite. A small size particle is produced
in the RBK coagulation model when two large particles collide with a slight size difference.
Dealing with the role played by large size particles in the creation of small size
particles was the major challenge of this research.

\medskip

\textbf{Keywords :} Coagulation; Fragmentation; Redner-Ben-Avraham-Kahng Coagulation Model; Existence.\\

\textbf{2020 Mathematics subject classification}: 34A12, 34K30, 34A34, 46B50, 26D07.

\section{Introduction}
The word ``particulate processes" is frequently used in mathematical models of systems involving a large number of individual particles or things that interact with one another. These systems can be used to simulate a wide range of physical, chemical, biological, and societal events.  Coagulation, fragmentation, annihilation, and growth are a few types of particulate processes, some of which are described below.

The dynamics of cluster formation involve the processes of coagulation and fragmentation, which determine how clusters combine to create larger ones and split into smaller fragments, respectively. In most fundamental coagulation models, the cluster size (or volume), which can be continuous \cite{Muller} or discrete \cite{Smoluchowski}, separates cluster particles based on physical parameters. Particle fragmentation occurs frequently in both natural and man-made processes. For instance, weathering and erosion can reduce rocks to smaller fragments, and chemical reactions can reduce larger molecules to smaller ones. A larger particle is reduced into smaller pieces or particles during this process. Different mechanisms, such as physical forces, chemical processes, or radiation, can cause this. The term "annihilation process" alludes to a fundamental physical process in which a particle collides with its antiparticle and mutually destroys. The annihilation process is used in particle physics to study particle properties and interactions, such as the Standard Model, upersymmetry, Dark Matter Annihilation, and Extra Dimensions. It can also describe the energy output of astronomical objects such as black holes, neutron stars, and quasars.

The coagulation-fragmentation type mean field models, of which Smoluchowski's coagulation (SC) system is a prototype, have drawn the attention of numerous researchers \cite{ball, stewart1, stewart2}. The modeling of the Smoluchowski coagulation equation is as follows: When a particle of size $j$, referred to as a $j$-cluster (a cluster composed of $j$ identical particles), interacts with a particle of size $k$, the resultant particle has a size of $(j+k)$. Due to the interaction of only two clusters, this reaction is binary. In this process, the average cluster size typically increases with time. The fundamental dynamic process of fragmentation, in contrast to coagulation, involves the phenomenon that a particular j-cluster splits into either two or more smaller clusters, where the average size of the cluster decreases over time. 

Let us now turn our attention to the mathematical formulation of the model in this study. For this, first consider a closed system of particles colliding in binary collisions in which the smaller particle completely annihilates and destroys the same quantity of the larger particle. This model is known as the Redner–Ben-Avraham–Kahng (RBK) coagulation model. In the RBK coagulation system, which is a cluster-eating model, two particles collide to form a single particle (similar to coagulation), but the resulting particle shrinks in size (similar to fragmentation). This model's working process is as follows: when a particle of size $i\ge 1$ collides with a particle of size $j\ge 1$, the resulting particle has size $|i-j|$. Throughout this paper, we will assume $i $ and $j$ are the integers and each particle of size $i \ge 1$ contains $i$ monomers. 

The finite dimensional RBK coagulation model was first proposed by Sidney Redner, Daniel ben-Avraham, and Bong June Kahng in 1984 \cite{Redner1} (also see \cite{Redner2}), and then further studies by da Costa et al. in \cite{dacosta1,dacosta2}, as an infinite RBK coagulation system. This is the infinite-dimensional ODE system that represents the rate of change of the concentration of cluster size $i$ with respect to time $t\ge 0$ and is written as
\begin{align}
\frac{df_i}{dt}&=\sum_{j=1}^{\infty}\theta_{i+j,j}f_{i+j} f_j-\sum_{j=1}^{\infty}\theta_{i,j}f_{i} f_j, \label{drbk}\\
f_i(0)&=f_i^{\text{in}}\ge 0, \label{drbkic}
\end{align}
for $i\ge 1$. Here, $\theta_{i,j}$ represents the rate of collision between particles of sizes $i\ge 1$ and $j\ge 1$ and is also referred to as the coagulation kernel, which is assumed to be positive and symmetric. Additionally, $f_i$ represents the concentration of particle size $i\ge 1$. In \eqref{drbk}, the first term on the right-hand side denotes the birth of particles of size $i$ which is produced when particles of sizes $(i+j)$ and $j$ collide, and the second term denotes the death of particles of size $i$ which is caused by the collision of particles of any sizes with those of size $i$.

Before we proceed to the coming section, it is necessary to define the moments of concentration (or number density function) to the discrete coagulation models. The $p^{th}$ moment of concentration $f=(f_i)_{i\ge 1}$, if it exists, is defined as:
\begin{align}\label{dmoment}
\mathcal{M}_p(f(t))=\mathcal{M}_p(t)=\sum_{i=1}^{\infty} i^p f_i(t), \ \ \ \ \ \ p\ge 0.
\end{align}

The total number of particles in the system at time $t$ is represented by the zeroth moment $\mathcal{M}_0(t)$. In the coagulation process, it is well known that the zeroth moment is a decreasing function with respect to time. Furthermore, the first moment $\mathcal{M}_1(t)$ signifies the total mass of the particles present in the system at time $t$. 

The RBK coagulation model and the Smoluchowski coagulation model \cite{Smoluchowski} are two distinct mathematical frameworks for describing particle coagulation or aggregation. Both seek to capture the behavior of particle interactions. The absence of the mass conservation property in \eqref{drbk}-\eqref{drbkic} the fundamental difference between RBK and Smoluchowski coagulation models. Unlike the Smoluchowski coagulation model, where mass is conserved over time, the RBK model exhibits a decrease in mass due to annihilation processes.

The RBK model has been explored in the context of understanding the dynamics of vicious civilizations in the study conducted by Ispolatov et al. \cite{Redner2}. Subsequently, da Costa et al. in their work \cite{dacosta1} established the existence and uniqueness of solutions for the infinite-dimensional discrete RBK coagulation model, considering a broad range of conditions on the coagulation coefficients. They also demonstrated the differentiability of solutions and their continuous dependence on the initial data. Notably, several significant invariance properties were proven in their study. Furthermore, a comprehensive analysis of the long-term behavior and scaling behavior of the solutions was conducted. The authors of \cite{dacosta2} specifically investigated the behaviour of the solutions to the discrete RBK coagulation system for the large-time. They focused on non-negative compactly supported input data and utilized the aforementioned invariance properties to transform the system into a finite-dimensional differential equation. The study of the continuous version of the RBK model has been explored in research papers \cite{ankik, ankik1}.

In this article, our main objective is to examine the global existence of solutions to the discrete RBK coagulation system \eqref{drbk}-\eqref{drbkic} by considering a particular class of coagulation rates that can be expressed in the following form:
\begin{eqnarray*}
\theta_{i,j}=\omega_i \omega_j+\kappa_{i,j}.
\end{eqnarray*}

The coagulation rates in the system satisfy $\kappa_{i,j} \le A \omega_i \omega_j$ for some $A>0$. It is important to emphasize that the sequence $(\omega_i)_{i \ge 1}$ made up of non-negative real numbers that can exhibit sub-linear, linear, or super-linear growth with respect to $i$. Further details regarding these assumptions are given by \eqref{ker1}-\eqref{ker3} in the next section. Specifically, the proofs of Theorems \ref{thm2.1} and \ref{thm2.2} are motivated from the proofs of existence results presented in \cite{Ali, dacosta1, Laurencot0}. In terms of appearance, the RBK coagulation model bears resemblance to the Smoluchowski coagulation model. However, the primary challenge in this study was addressing the contribution of large particles to the formation of small particles within the RBK coagulation model, which occurs when two large particles collide with only a slight disparity in size.

The paper is structured as follows: In Section 2, we provide the definition of a solution and establish the necessary assumptions on the coagulation coefficients. This section also outlines the main results of the article, namely Theorems \ref{thm2.1} and \ref{thm2.2}. Section 3 presents preliminary results that will be utilized in the proofs of the main theorems. Sections 4 and 5 are dedicated to proving Theorems \ref{thm2.1} and \ref{thm2.2}, respectively.


\section{Main results}

To conduct the mathematical analysis of equations \eqref{drbk}-\eqref{drbkic}, it is necessary to carefully consider the appropriate function spaces that are suitable for studying the system. We shall take into account the Banach space $X_p$, as is common in works in this field and defined by
\begin{align*}
X_p&:=\{x=(x_i)_{i\ge1}:\|x\|_p <\infty\|,\ \ \ p\ge 0,
\end{align*}
with the norm
\begin{align*}
\|x\|_p&:= \sum_{i=1}^{\infty}i^p|x_i|.
\end{align*}
Then we set
\begin{align*}
X_p^+:=\{x\in X_p, x_i\ge 0 \text{ for each } i\ge 1\},
\end{align*}
which is the positive cone of the space $X_p$.

Throughout this paper, we assume that the coagulation kernel $(\theta_{i,j})$ is non-negative, symmetric (i.e. $0 \le \theta_{i,j} =\theta_{j,i}, \forall i,j \ge 1$) and satisfying
\begin{eqnarray}
\theta_{i,j}=\omega_i \omega_j+\kappa_{i,j},\label{ker1}
\end{eqnarray}

where the sequences $(\omega_i)_{i\ge 1}$ and $(\kappa_{i,j})_{i,j \ge 1}$ contains non-negative real numbers. In addition, we assume that $(\omega_i)$ and $(\kappa_{i,j})$ satisfy

\begin{eqnarray}
\lim_{i\to \infty} \frac{\omega_i}{i}=0,~~~\text{and}~~~ \lim_{j\to \infty} \frac{\kappa_{i,j}}{j}=0,~~\forall ~i\ge 1\label{ker2}
\end{eqnarray}
or
\begin{eqnarray}
\inf_{i\ge 1}\frac{\omega_i}{i}=R>0,~~~\text{and}~~~\kappa_{i,j}\le A\omega_i \omega_j, ~~\forall ~i\ge 1. \label{ker3}
\end{eqnarray}

Let us now define the notion of solution to \eqref{drbk}-\eqref{drbkic}.

\begin{definition}\label{defsol}
Consider $T\in (0, \infty)$ and $f^{\text{in}}=(f_i^{\text{in}})_{i \ge 1} \in X_p^+$. A solution $f=(f_i){i\ge 1}$ to equations \eqref{drbk}-\eqref{drbkic} on the interval $[0,T)$ with initial data $f(0)=f^{\text{in}}$ can be defined as a function $f:[0,T] \to X_p^+$ satisfying the following conditions:
\begin{enumerate}
\item for all $i \ge 1$, the function $f_i:[0,T] \to \mathbb{R}$ is continuous and the supremum of $|f(t)|_p$ over $t\in[0,T]$ is finite,\\

\item for all $i \ge 1$ and $t\in[0,T]$, the series $\sum_{j=1}^{\infty} \theta_{i,j} f_j$ belongs to $L^1(0,t)$,\\

\item for all $i \ge 1$ and $t\in[0,T]$, the equation governing $f_i(t)$ is given by: 
$$f_i(t)=f_i^{\text{in}}+\int_0^t\bigg(\sum_{j=1}^{\infty}\theta_{i+j,j}f_{i+j} f_j-\sum_{j=1}^{\infty}\theta_{i,j}f_{i} f_j\bigg)ds.$$
\end{enumerate}
\end{definition}

Let us now state the main results of the paper, which generalize \cite[Theorem 3.1]{dacosta1} to a broader class of coagulation rates given by \eqref{ker1}-\eqref{ker3}.

\begin{theorem}\label{thm2.1}
Let us assume \eqref{ker1} holds along with either \eqref{ker2} or \eqref{ker3} for coagulation kernel $\theta_{i,j}$. For any initial data $f^{\text{in}}=(f_i^{\text{in}})_{i\ge 1} \in X_1^+$, the system \eqref{drbk}-\eqref{drbkic} has at least one solution $f$ on $[0, +\infty)$ which satisfies $f(t) \in X_1^+$ for each $t\in [0, +\infty)$, and
\begin{eqnarray}\label{mass1}
\|f(t)\|_1 \le \|f^{\text{in}}\|_1.
\end{eqnarray}
\end{theorem}

\begin{theorem}\label{thm2.2}
Let us assume that \eqref{ker1} and \eqref{ker3} hold for coagulation kernel $\theta_{i,j}$ and initial data satisfies $f^{\text{in}}=(f_i^{\text{in}})_{i\ge 1} \in X_0^+$. Then the system \eqref{drbk}-\eqref{drbkic} has at least one solution $f$ on $[0, +\infty)$ which satisfies $f(t) \in X_0^+$ for each $t\in [0, +\infty)$, and
\begin{eqnarray}\label{mass2}
\|f(t)\|_1 <\infty, \ \ \ \ \ t>0.
\end{eqnarray}
\end{theorem}

\section{Preliminaries}

It is well known that in the RBK coagulation model, the only process is coagulation and the resulting cluster is smaller in size, no clusters larger than $n$ can be produced. For fixed $n$. This condition can be mathematically represented by the following truncated system of $n$ ordinary differential equations. First, we fix $n\ge 2$ and let us consider an initial condition $f_j(0) = 0$ for $j > n$, where $n$ is a positive integer.

\begin{eqnarray}
\frac{df^n_i}{dt}&=&\sum_{j=1}^{n-i}\theta^n_{i+j,j}f^n_{i+j} f^n_j-\sum_{j=1}^{n}\theta^n_{i,j}f^n_{i} f^n_j, \ \ \ \ \ \ i\in \{1,...,n\} \label{tdrbk}\\
f^n_i(0)&=&f_i^{n,\text{in}}\ge 0\label{tdrbkic}.
\end{eqnarray}

\begin{lemma}\label{l1}
For each $n\ge 2$, there exists a unique non-negative solution $(f^n_i)_{1\le i \le n} $
to the system \eqref{tdrbk}-\eqref{tdrbkic} in the space $C^1([0,+\infty), \mathbb{R}^n)$. Additionally, for $t\in[0,+\infty)$, we have
\begin{eqnarray*}
\sum_{i=1}^{n} if^n_i(t)\le \sum_{i=1}^{n} if_i^{n,\text{in}}.
\end{eqnarray*}
\end{lemma}

The given system \eqref{tdrbk}-\eqref{tdrbkic} represents a finite-dimensional system, with the right-hand side being a polynomial function of the solution vector $f^n=(f^n_i(\cdot))$. By employing the Picard-Lindel\"of existence theorem, we can directly establish the existence of local solutions to the associated Cauchy problems. The following lemma provides a concise summary of the significant findings concerning solutions to this finite-dimensional system.

\begin{lemma}\label{l2}
Let $f^n=(f_i^n(.)): I_{\text{max}} \to \mathbb{R}^n_+$ be the unique local solution of \eqref{tdrbk}-\eqref{tdrbkic} with $I_{\text{max}}$ as its maximal interval. Then for $0 \le t_1 \le t_2 \le +\infty$
\begin{enumerate}
\item[($i$)] For every sequence of non-negative real numbers $(\psi_i)_{1\le i \le n}$ and $1\le m \le n$ the following holds true
\begin{eqnarray}\label{wdrbk}
\sum_{i=m}^{n}\psi_i f_i^n(t_2)-\sum_{i=m}^{n}\psi_i f_i^n(t_1)&=&\int_{t_1}^{t_2}\bigg[-\sum_{\mathcal{T}_1(m,n)}(\psi_i-\psi_{i-j})\theta_{i,j}f^n_{i}(s) f^n_j(s)\nonumber\\
& & \ \ \ \ \ \  -\sum_{\mathcal{T}_2(m,n)}\psi_i \theta_{i,j}f^n_{i}(s) f^n_j(s)\bigg]ds,
\end{eqnarray}
where
\begin{eqnarray*}
\mathcal{T}_1(m,n)&=& \{(i,j)\in \{m,...,n\}\times\{1,...,n\};j\le i-m\},\\
\mathcal{T}_2(m,n)&=& \{(i,j)\in \{m,...,n\}\times\{1,...,n\};j\ge i-m+1\}.
\end{eqnarray*}
\item[($ii$)] for initial condition given in \eqref{tdrbkic}. if all components $f_i^{n,\text{in}}$ are non-negative then all components of solution $f^n=(f^n_i)_{1\le i \le n}$ are also non-negative.
\item[($iii$)] $\sup I_{\text{max}} =+\infty$.
\end{enumerate}

\end{lemma}

The proof of Lemma \ref{l2} follows from \cite[Proposition 2.3]{dacosta1}.
Put $m=1$ in \eqref{wdrbk}, we obtain
\begin{eqnarray*}
\mathcal{T}_1(1,n)&=& \{(i,j)\in \{1,...,n\}\times\{1,...,n\};j\le i-1\},\\
\mathcal{T}_2(1,n)&=& \{(i,j)\in \{1,...,n\}\times\{1,...,n\};j\ge i\},
\end{eqnarray*}
and
\begin{eqnarray}\label{wdrbk1}
\sum_{i=1}^{n}\psi_i f_i^n(t_2)-\sum_{i=1}^{n}\psi_i f_i^n(t_1)=\int_{t_1}^{t_2}\bigg[\sum_{i=1}^{n}\sum_{j=1}^{i-1}(\psi_{i-j}-\psi_i-\psi_j)\theta_{i,j}f^n_{i}(s) f^n_j(s)\bigg]ds.
\end{eqnarray}

As a result of \eqref{wdrbk}-\eqref{wdrbk1}, we develop various estimates in the following sequence of lemmas, which are required to prove Theorem \ref{thm2.1} and Theorem \ref{thm2.2}. The proof of following lemma is motivated by \cite[Lemma 3.2]{Laurencot0}.

\begin{lemma}\label{l3}
Suppose that \eqref{ker1} holds and $f^n=(f^n_i)_{1\le i \le n}$ be a solution of \eqref{tdrbk}-\eqref{tdrbkic}. Then the following holds:
\begin{enumerate}
\item[(i)] For $0 \le t_1 \le t_2 < +\infty$,
\begin{eqnarray}
\sum_{i=1}^{n} if^n_i(t_2) \le \sum_{i=1}^{n} if^n_i(t_1)\le \sum_{i=1}^{n} if_i^{\text{in}},\ \ \ \ 0 \le t_1 \le t_2 < +\infty. \label{eqmass2}
\end{eqnarray}
\item[(ii)] For $t\in[0,+\infty)$,
\begin{eqnarray}
\sum_{i=1}^{n} f_i^n(t)+\frac{1}{2}\int_{0}^{t}\bigg|\sum_{i=1}^{n} \omega_i f_i^n(s)\bigg|^2 ds\le \sum_{i=1}^{n} f_i^{\text{in}}. \label{eqnum}
\end{eqnarray}
\item[(iii)] If $1 \le r \le n$ and $0 \le t_1 \le t_2 < +\infty$, then
\begin{eqnarray}
\int_{t_1}^{t_2}\bigg(\sum_{i=r}^{n} \omega_i f_i^n(s)\bigg)^2 ds \le 2r^{-1/2}\sum_{i=1}^{n} i^{1/2} f_i^n(t_1),\ \ \ 0 \le t_1 \le t_2 < +\infty. \label{eqL2}
\end{eqnarray}

\end{enumerate}
\end{lemma}
\begin{proof}
\begin{enumerate}
\item[(i)] First, let $\psi_i=i$ in \eqref{wdrbk1}, we conclude
\begin{eqnarray*}
\sum_{i=1}^{n} i f_i^n(t_2)-\sum_{i=1}^{n} i f_i^n(t_1)= -2 \int_{t_1}^{t_2}\sum_{i=1}^{n}\sum_{j=1}^{i-1} j \theta_{i,j}f_i^n(s)f_j^n(s) ds \le 0.
\end{eqnarray*}
which implies \eqref{eqmass2}.
\item[(ii)] Next, take $\psi_i=1$, $t_2=t$ and $t_1=0$ in \eqref{wdrbk1}, we obtain
\begin{eqnarray*}
\sum_{i=1}^{n} f_i^n(t)-\sum_{i=1}^{n} f_i^n(0)= - \int_{0}^{t}\sum_{i=1}^{n}\sum_{j=1}^{i-1} \theta_{i,j}f_i^n(s)f_j^n(s) ds.
\end{eqnarray*}
Now, by using change of order of summation, the following equality holds
\begin{eqnarray*}
\int_{0}^{t}\sum_{i=1}^{n}\sum_{j=1}^{i-1} \theta_{i,j}f_i^n(s)f_j^n(s) ds &=& \frac{1}{2} \int_{0}^{t}\sum_{i=1}^{n}\sum_{j=1}^{i-1} \theta_{i,j}f_i^n(s)f_j^n(s) ds +\frac{1}{2} \int_{0}^{t}\sum_{i=1}^{n}\sum_{j=i}^{n} \theta_{i,j} f_i^n(s)f_j^n(s) ds\\
&=& \frac{1}{2} \int_{0}^{t}\sum_{i=1}^{n}\sum_{j=1}^{n} \theta_{i,j}f_i^n(s)f_j^n(s) ds.
\end{eqnarray*}
Combining preceding two equalities, we obtain the following equality
\begin{eqnarray*}
\sum_{i=1}^{n} f_i^n(t) + \frac{1}{2} \int_{0}^{t}\sum_{i=1}^{n}\sum_{j=1}^{n} \theta_{i,j}f_i^n(s)f_j^n(s) ds = \sum_{i=1}^{n} f_i^n(0),
\end{eqnarray*}
which then along with the inequality $\omega_i \omega_j \le \theta_{i,j}$ gives \eqref{eqnum}.

\item[(iii)] Now, in \eqref{wdrbk1}, put $\psi_i=i^\frac{1}{2}$ and by using the inequality  $(i-j)^{1/2}\le i^{1/2} \implies (i-j)^{1/2}- i^{1/2}\le 0$, we obtain

\begin{eqnarray}
\sum_{i=1}^{n}i^{1/2}f_i^n(t_2)-\sum_{i=1}^{n}i^{1/2}f_i^n(t_1)\le -\int_{t_1}^{t_2} \sum_{i=1}^{n}\sum_{j=1}^{i-1} j^{1/2}\theta_{i,j}f_i^n(s)f_j^n(s) ds.\label{*}
\end{eqnarray}

Next, Changing the order of summation on the right-hand side of \eqref{*}, we obtain
\begin{eqnarray*}
\sum_{i=1}^{n}\sum_{j=1}^{i-1} j^{1/2}\theta_{i,j}f_i^n(s)f_j^n(s)&=&\frac{1}{2} \sum_{i=1}^{n}\sum_{j=1}^{i-1} j^{1/2}\theta_{i,j}f_i^n(s)f_j^n(s)+\frac{1}{2}\sum_{j=1}^{n}\sum_{i=j}^{n} i^{1/2}\theta_{i,j}f_i^n(s)f_j^n(s)\\
&=&\frac{1}{2} \sum_{i=1}^{n}\sum_{j=1}^{n} \min\{i,j\}^{1/2}\theta_{i,j}f_i^n(s)f_j^n(s).
\end{eqnarray*}

Therefore, by using above equality in \eqref{*}, we get
\begin{eqnarray*}
\sum_{i=1}^{n}i^{1/2}f_i^n(t_2)+\frac{1}{2}\int_{t_1}^{t_2} \sum_{i=1}^{n}\sum_{j=1}^{n} \min\{i,j\}^{1/2}\theta_{i,j}f_i^n(s)f_j^n(s) ds \le \sum_{i=1}^{n}i^{1/2}f_i^n(t_1),
\end{eqnarray*}
which gives
\begin{eqnarray}
& &\sum_{i=1}^{n}i^{1/2}f_i^n(t_2)\le \sum_{i=1}^{n}i^{1/2}f_i^n(t_1),\nonumber\\
& & \text{and}\nonumber\\
& &\frac{1}{2}\int_{t_1}^{t_2} \sum_{i=1}^{n}\sum_{j=1}^{n} \min\{i,j\}^{1/2}\theta_{i,j}f_i^n(s)f_j^n(s) ds \le \sum_{i=1}^{n}i^{1/2}f_i^n(t_1).\label{**}
\end{eqnarray}
Next, by using \eqref{**} and $\omega_i \omega_j \le \theta_{i,j}$, we obtain, for $1 \le r_1 \le r_2 \le n$,
\begin{eqnarray*}
r_1^{1/2}\int_{t_1}^{t_2}\sum_{i=r_1}^{r_2}\sum_{j=r_1}^{r_2} \omega_i\omega_j f_i^n(s)f_j^n(s) ds &\le & \int_{t_1}^{t_2}\sum_{i=r_1}^{r_2}\sum_{j=r_1}^{r_2} \min\{i,j\}^{1/2} \omega_i\omega_j f_i^n(s)f_j^n(s) ds \\
&\le &\int_{t_1}^{t_2} \sum_{i=1}^{n}\sum_{j=1}^{n} \min\{i,j\}^{1/2}\theta_{i,j}f_i^n(s)f_j^n(s) ds\\
&\le & 2\sum_{i=1}^{n}i^{1/2}f_i^n(t_1),
\end{eqnarray*}
which implies
\begin{eqnarray}
\int_{t_1}^{t_2}\bigg(\sum_{i=r_1}^{r_2} \omega_i f_i^n(s)\bigg)^2 ds \le 2r_1^{-1/2}\sum_{i=1}^{n}i^{1/2}f_i^n(t_1).\label{3*}
\end{eqnarray}

\end{enumerate}
 Hence Lemma \ref{l3} is proved.
\end{proof}

Another implication of \eqref{wdrbk} is the following lemma, which applies to the kernel satisfying \eqref{ker1} along with \eqref{ker3}. 
\begin{lemma}\label{l4}
Under the assumption of \eqref{ker1} and \eqref{ker3} for the coagulation kernel $\theta_{i,j}$, the following statement holds for every $t\in [0, +\infty)$:
\begin{eqnarray}\label{lt2}
\sum_{i=1}^{n} i f_i^n(t)\le \frac{2}{R} \bigg( \sum_{i=1}^{n} f_i^{\text{in}}\bigg)^{1/2} t^{-1/2}.
\end{eqnarray}
\end{lemma}

\begin{proof}
First, by using \eqref{ker3} and \eqref{eqnum}, we conclude that
\begin{eqnarray*}
\int_0^t \bigg| \sum_{i=1}^{n} i f_i^n(s) \bigg|^2 ds & \le & \int_0^t \frac{1}{R^2}\bigg| \sum_{i=1}^{n} \omega_i f_i^n(s) \bigg|^2 ds\\
& \le & \frac{2}{R^2} \int_0^t \sum_{i=1}^{n} f_i^{\text{in}}(s) ds.
\end{eqnarray*}
Since $s\to \sum_{i=1}^{n} i f_i^n(s)$ is non-increasing by \eqref{eqmass2}, so above inequality implies
$$ t\bigg|\sum_{i=1}^n if_i^n (t)\bigg|^2 \le \frac{2}{R^2}\sum_{i=1}^{n} f_i^{ \text{in}}.$$
This completes the proof of Lemma \ref{l4}.
\end{proof}

\section{Proof of Theorem \ref{thm2.1}}
We now proceed to prove Theorem \ref{thm2.1}. First, we need to demonstrate that the sequence of solutions to finite dimensional truncated system \eqref{tdrbk}-\eqref{tdrbkic} is compact. To prove this, we will employ an alternative method compared to the one used in \cite{dacosta1}, which relies on Helly's selection theorem. Therefore, one can establish the required compactness using the above mentioned theorems. For the detailed proof of Theorem \ref{thm2.1}, from \eqref{ker1}-\eqref{ker2}, we can easily say that for all $i,j \ge 1$
$$\kappa_{i,j} \le C_1(ij)^{\alpha},~~~~\text{for some}~\alpha \in [0,1),$$ 
where $C_1$ is a constant. The kernels satisfying above condition also satisfy $$\kappa_{i,j} \le C_1(ij).$$
It follows a similar approach as in \cite[Theorem 3.1]{dacosta1}, so the proof of existence of solution under the assumptions \eqref{ker1}-\eqref{ker2} will be omitted here.

Now, we turn to the case where the coagulation rate $\theta_{i,j}$ satisfies \eqref{ker3}. It follows from \eqref{ker3} that, for each $i\ge 1$ and $n\ge i$, we have
\begin{eqnarray}\label{eqsimp}
\sum_{j=r_1}^{r_2} \theta_{i,j} f_j^n \le (1+A)\omega_i \sum_{j=r_1}^{r_2} \omega_j f_j^n, \ \ \ \ \ 1\le r_1 \le r_2 \le n.
\end{eqnarray}

In order to use the Helly's selection theorem, it is necessary to demonstrate that the sequence $(f_i^n)_{n>i}$ is uniformly bounded in $W^{1,1}(0,T)$, for all fixed $T\in (0,\infty)$. This can be shown by using the following lemma:

\begin{lemma}\label{l3.1}
For all $i\geq 1$, and for every $T\in (0,+\infty)$, there exists a constant $\Pi$ that depends on $f^{\text{in}}=(f_i^{\text{in}})_{i \ge 1}$ such that
\begin{eqnarray}\label{normW12}
 \|f_i^n\|_{L^1(0,T)} +\bigg\|\dfrac{df_i^n}{dt}\bigg\|_{L^1(0,T)}\le \Pi,\ \ \ \ t\in [0,T].
\end{eqnarray}
\end{lemma}
\begin{proof}
 We infer from \eqref{eqnum} that, for $T\in (0, +\infty)$, 
$$\sum_{i=1}^n f_i^n(t) \le \sum_{i=1}^n f_i^{\text{in}},$$ 
which gives
\begin{eqnarray}
f_i^n(t) \le \sum_{j=1}^{n}f_j^n(t) \le \sum_{j=1}^{n}f_j^{\text{in}} \le \sum_{j=1}^{\infty}f_j^{\text{in}}\le \|f^{\text{in}}\|_0, \ \ \ \forall \ t\in[0,T]. \label{normb1}
\end{eqnarray} 

Next, by using \eqref{tdrbk} and \eqref{eqsimp}, we obtain

\begin{eqnarray*}
\bigg|\dfrac{df_i^n}{dt}\bigg| &=& \bigg|\sum_{j=1}^{n-i}\theta_{i+j,j}f^n_{i+j} f^n_j-\sum_{j=1}^{n}\theta_{i,j}f^n_{i} f^n_j \bigg|\\
&\le& (1+A)\bigg|\sum_{j=1}^{n-i}\omega_{i+j} \omega_{j}f^n_{i+j} f^n_j-\sum_{j=1}^{n}\omega_{i} \omega_{j}f^n_{i} f^n_j \bigg|\\
&\le & (1+A)\bigg|\sum_{j=1}^{n-i} \omega_{i+j} f^n_{i+j}\bigg|\bigg|\sum_{j=1}^{n-i}\omega_{j} f^n_j\bigg|+(1+A) \bigg|\sum_{j=1}^{n}\omega_{j} f^n_j \bigg|^2\\
&\le &2(1+A) \bigg|\sum_{j=1}^{n} \omega_{j} f^n_j(s)\bigg|^2.
\end{eqnarray*}

Finally, by using \eqref{eqnum}, we conclude that for each $T\in (0, +\infty)$ and $i \ge 1$
\begin{eqnarray*}
\bigg\|\dfrac{df_i^n}{dt}\bigg\|_{L^1(0,T)} &=& \int_0^T \bigg|\dfrac{df_i^n}{dt}\bigg| ds\\
 &\le & 2(1+A) \int_0^T\bigg|\sum_{j=1}^{n} \omega_{j} f^n_j(s)\bigg|^2ds\\
 & \le & 4(1+A) \sum_{i=1}^{n} f_i^{\text{in}}.
\end{eqnarray*}
The above inequality along with \eqref{normb1} gives \eqref{normW12}. This completes the proof of Lemma \ref{l3.1}.
\end{proof}

As a consequence, for every $T \in (0, +\infty)$, the sequence $(f_i^n)_{n\ge i}$ is bounded in $W^{1,1}(0,T)$. Due to the fact that $W^{1,1}(0,T)$ is a closed subspace of $BV(0,T)$, Helly's selection theorem concludes that there exists a subsequence $(f_i^{n_k})$ of $(f_i^n)_{n \ge i}$ converging pointwise to a $BV$ function $(f_i)$ in $([0,T])$ for each $i \ge 1$ and $T \in (0, +\infty)$ i.e.

\begin{eqnarray}\label{conv1}
f_i^{n_k}(t) \to f_i(t) \ \ \ \mbox{as} \ k \to \infty, \ \forall t\in [0,T], \forall i\ge 1.
\end{eqnarray}

Now, we check that $f(t) \in X_1^+$ for all $t \in (0,+\infty)$. For $n_k \ge r$, \eqref{eqmass2} implies that if $t \in (0, +\infty)$ and $r \ge 1$, then
$$\sum_{i=1}^{r} i f_i^{n_k}(t) \le \|f^{\text{in}}\|_1.$$
Now, we use \eqref{conv1}, to pass the limit $k \to +\infty$ in preceding inequality to get
$$\sum_{i=1}^{r} i f_i(t) \le \|f^{\text{in}}\|_1.$$
Given that $r \ge 1$ is arbitrarily chosen, the preceding inequality implies that
$$\sum_{i=1}^{\infty} i f_i(t) \le \|f^{\text{in}}\|_1,$$
which guarantees that $f(t) \in X_1^+$ and \eqref{mass1} holds.

We will now demonstrate the convergence of the truncated system \eqref{tdrbk}-\eqref{tdrbkic} to the RBK coagulation system \eqref{drbk}-\eqref{drbkic}. For this, the following lemma is essentially required:

\begin{lemma}\label{Convl2}
Assume that \eqref{ker1} and \eqref{ker3} hold and $(f_i^{n_k})$ and $(f_i)$ are as in \eqref{conv1}. Then, for each $i \ge 1$ and $T \in (0, +\infty)$, we have
\begin{align}\label{convL2}
 \sum_{j=1}^{n_k} \theta_{i,j} f_{j}^{n_k} \to \sum_{j=1}^{\infty} \theta_{i,j} f_{j} \ \ \text{in} \ L^1(0,T),
\end{align}
as $k\to \infty $.
\end{lemma}
\begin{proof}
It can easily be followed from \eqref{eqmass2}, \eqref{eqL2} and \eqref{eqsimp} that for $i \ge 1$, $T\in (0, +\infty)$ and $1\le r_1 \le r_2 \le n_k$
\begin{align}\label{eqfL2}
\int_0^T \bigg|\sum_{j=r_1}^{r_2} \theta_{i,j} f_{j}^{n_k}(s)\bigg|^2 ds \le 2(1+A)^2 \omega_i^2 r_1^{-1/2} \|f^{\text{in}}\|_1.
\end{align}
 We now use \eqref{conv1} to pass the limit as $k \to +\infty$ in \eqref{eqfL2} to obtain
\begin{align}\label{eqf1L2}
\int_0^T \bigg|\sum_{j=r_1}^{r_2} \theta_{i,j} f_{j}(s)\bigg|^2 ds \le 2(1+A)^2 \omega_i^2 r_1^{-1/2} \|f^{\text{in}}\|_1 \ \ \ 1\le r_1 \le r_2 < +\infty.
\end{align}
Since the right-hand side of \eqref{eqf1L2} does not depend on $r_2$ and $r_2$ is arbitrary, so for each $i \ge 1$ and $T \in (0, +\infty)$
\begin{align}\label{eqf2L2}
\int_0^T \bigg|\sum_{j=r}^{\infty} \theta_{i,j} f_{j}(s)\bigg|^2 ds \le 2(1+A)^2 \omega_i^2 r^{-1/2} \|f^{\text{in}}\|_1, \ \ \ \   r \ge 1 ,
\end{align}
which implies $\sum_{j=r}^{\infty} \theta_{i,j} f_{j}(s) \in L^2(0,T)$ for each $r \ge 1$.
Since, from Cauchy–Schwarz inequality, we conclude
\begin{align*}
	\int_0^T \bigg|\sum_{j=r}^{\infty} \theta_{i,j} f_{j}(s)\bigg| ds &\le \bigg[\int_0^T \bigg|\sum_{j=r}^{\infty} \theta_{i,j} f_{j}(s)\bigg|^2ds\bigg]^{1/2} \bigg[ \int_0^T 1^2 ds\bigg]^{1/2} \\
	&\le \sqrt{2T}(1+A) \omega_i r^{-1/4} \|f^{\text{in}}\|_1^{1/2}, \ \ \ \ \  r \ge 1.
\end{align*}
This implies $\sum_{j=r}^{\infty} \theta_{i,j} f_{j}(s) \in L^1(0,T)$ for each $r \ge 1$ and so condition (2) of Definition \ref{defsol} is fulfilled.

Next, for each $n_k \ge r \ge i$, from \eqref{eqfL2} and \eqref{eqf2L2}, we obtain

\begin{align*}
\bigg\| \sum_{j=1}^{n_k} \theta_{i,j} f_j^{n_k}- \sum_{j=1}^{\infty} \theta_{i,j} f_j \bigg\|_{L^1(0,T)}  \le & \bigg\| \sum_{j=1}^{r-1} \theta_{i,j} (f_j^{n_k}- f_j) \bigg\|_{L^1(0,T)} \\
& +\bigg\| \sum_{j=r}^{n_k} \theta_{i,j} f_j^{n_k} \bigg\|_{L^1(0,T)} +\bigg\| \sum_{j=r}^{\infty} \theta_{i,j} f_j \bigg\|_{L^1(0,T)} \\
\le & \sum_{j=1}^{r-1} \theta_{i,j} \|f_j^{n_k}- f_j\|_{L^1(0,T)}+ 2\sqrt{2T}(1+A) \omega_i \|f^{\text{in}}\|_1^{1/2} r^{-1/4}.
\end{align*}

Because $\theta_{i,j}$ is bounded in the interval $[0, r)$, \eqref{conv1} implies that
\begin{align*}
\limsup_{k \to \infty} \bigg\| \sum_{j=1}^{n_k} \theta_{i,j} f_j^{n_k}- \sum_{j=1}^{\infty} \theta_{i,j} f_j \bigg\|_{L^1(0,T)} \le 2\sqrt{2T}(1+A) \omega_i \|f^{\text{in}}\|_1^{1/2} r^{-1/4}.
\end{align*}

Above equation is true for all $r \ge 1$ and the left-hand side is independent of $r$. Therefore, by passing the limit $r \to +\infty$, we complete the proof of Lemma \ref{Convl2}.
\end{proof}

It is easy to conclude from \eqref{3*} that for $T\in (0, +\infty)$ and $1\le r_1 \le r_2 \le n_k$
\begin{align*}
	\int_0^T \bigg|\sum_{j=r_1}^{r_2} \omega_{j} f_{j}^{n_k}(s)\bigg|^2 ds \le 2 r_1^{-1/2} \|f^{\text{in}}\|_1.
\end{align*}
We now use \eqref{conv1} to pass the limit as $k \to +\infty$ to obtain
\begin{align*}
	\int_0^T \bigg|\sum_{j=r_1}^{r_2} \omega_{j} f_{j}(s)\bigg|^2 ds \le 2 r_1^{-1/2} \|f^{\text{in}}\|_1 \ \ \ 1\le r_1 \le r_2 < +\infty.
\end{align*}
Since, $r_2$ is arbitrary and the right-hand side of \eqref{eqf1L2} does not depend on $r_2$, so for each $T \in (0, +\infty)$
\begin{align}\label{eqf2L2*}
	\int_0^T \bigg|\sum_{j=r}^{\infty} \omega_{j} f_{j}(s)\bigg|^2 ds \le 2 r^{-1/2} \|f^{\text{in}}\|_1, \ \ \ \   r \ge 1 ,
\end{align}
which implies $\sum_{j=r}^{\infty} \omega_{j} f_{j}(s) \in L^2(0,T)$ for each $r \ge 1$.

Now, we show how each summation on the right-hand side of \eqref{tdrbk}-\eqref{tdrbkic}  converges to the corresponding summation on the right-hand side of \eqref{drbk}-\eqref{drbkic}.  In order to do this, we must demonstrate the following results:
\begin{align}\label{1term}
\sum_{j=1}^{n_k-i} \theta_{i+j,j} f_{i+j}^{n_k} f_j^{n_k} \xrightarrow[k \to \infty]{} \sum_{j=1}^{\infty} \theta_{i+j,j} f_{i+j} f_j \ \ \ \ \mbox{in} \ L^1(0,T),
\end{align}
and
\begin{align}\label{2term}
\sum_{j=1}^{n_k} \theta_{i,j} f_{i}^{n_k} f_j^{n_k} \xrightarrow[k \to \infty]{} \sum_{j=1}^{\infty} \theta_{i,j} f_{i} f_j\ \ \ \ \mbox{in} \ L^1(0,T).
\end{align}

Let us first prove the estimate \eqref{1term}. For each $T\in (0, +\infty)$ and $n_k \ge r > i$, by using \eqref{ker1}, \eqref{ker3} and \eqref{eqf2L2*}, we obtain
\begin{align*}
& \bigg\|\sum_{j=1}^{n_k-i} \theta_{i+j,j} f_{i+j}^{n_k} f_j^{n_k}- \sum_{j=1}^{\infty} \theta_{i+j,j} f_{i+j} f_j \bigg\|_{L^1(0,T)}\\
\le & \bigg\|\sum_{j=1}^{r-i} \theta_{i+j,j} f_{i+j}^{n_k} f_j^{n_k}- \sum_{j=1}^{r-i} \theta_{i+j,j} f_{i+j} f_j \bigg\|_{L^1(0,T)}\\
& +\bigg\|\sum_{j=r-i+1}^{n_k-i} \theta_{i+j,j} f_{i+j}^{n_k} f_j^{n_k} \bigg\|_{L^1(0,T)} +\bigg\|\sum_{j=r-i+1}^{\infty} \theta_{i+j,j} f_{i+j} f_j \bigg\|_{L^1(0,T)} \\
\le & \int_0^T\sum_{j=1}^{r-i} \theta_{i+j,j} |f_{i+j}^{n_k} f_j^{n_k}- f_{i+j} f_j|ds\\
& +(1+A)\int_0^T\bigg|\sum_{j=r-i+1}^{n_k-i} \omega_{i+j} \omega_{j} f_{i+j}^{n_k} f_j^{n_k} \bigg|ds +(1+A)\int_0^T\bigg|\sum_{j=r-i+1}^{\infty} \omega_{i+j} \omega_{j} f_{i+j} f_j \bigg|ds \\
\le & \int_0^T\sum_{j=1}^{r-i} \theta_{i+j,j} |f_{i+j}^{n_k} f_j^{n_k}- f_{i+j} f_j|ds\\
& +(1+A)\int_0^T\bigg|\sum_{j=r-i+1}^{n_k} \omega_{j} f_j^{n_k} \bigg|^2ds +(1+A)\int_0^T\bigg|\sum_{r-i+1}^{\infty} \omega_{j} f_j \bigg|^2ds\\
\le & \int_0^T\sum_{j=1}^{r-i} \theta_{i+j,j} |f_{i+j}^{n_k} f_j^{n_k}- f_{i+j} f_j|ds+ 4 (r-i+1)^{-1/2} \|f^{\text{in}}\|_1.
\end{align*}
Since, $\theta_{i+j,j}$ is bounded for $ 1\le j \le r-i$, we use \eqref{conv1} to pass the limit $k \to \infty$ in the previous inequality and get
$$\limsup_{k \to \infty} \bigg\|\sum_{j=1}^{n_k-i} \theta_{i+j,j} f_{i+j}^{n_k} f_j^{n_k}- \sum_{j=1}^{\infty} \theta_{i+j,j} f_{i+j} f_j \bigg\|_{L^1(0,T)} \le 4 (r-i+1)^{-1/2} \|f^{\text{in}}\|_1.$$ 
Above equation is true for all $r > i$ and the left-hand side is independent of $r$. Therefore, by passing the limit $r \to +\infty$, we can easily conclude \eqref{1term}.

As pointed out above, the proof of \eqref{2term} is entirely analogous and will be omitted.
Owing to \eqref{1term} and \eqref{2term}, it is straightforward to pass to the limit in the $i$-th equation of \eqref{tdrbk} as $k\to +\infty$ after replacing $n$ by $n_k$.

\begin{corollary}\label{ext5}
	The solution obtained above can be extended to $t\in [0,\infty)$ as an admissible solution.
\end{corollary}
The proof of Corollary \ref{ext5} is subsequent to the one provided in \cite[Corollary 3.2]{dacosta1}.Consequently, we can omit it and establish that $f$ is a solution to \eqref{drbk}-\eqref{drbkic} on $[0, +\infty)$. Hence, the proof of Theorem \ref{thm2.1} is complete.

\section{Proof of Theorem \ref{thm2.2}}
In this section, our goal is to establish the existence of at least one solution to the RBK coagulation system \eqref{drbk}-\eqref{drbkic} in the $X_0$-norm by assuming that the coagulation kernel $(\theta_{i,j})$ satisfies \eqref{ker1} and \eqref{ker3} for all $i,j\ge 1$. We will utilize the Helly's selection theorem to demonstrate the compactness of the sequence of solutions to \eqref{tdrbk}-\eqref{tdrbkic}. To accomplish this, we need to prove Lemma \ref{l3.1} for an initial condition in the $X_0$-norm. It is evident that in Lemma \ref{l3.1}, $\Pi$ does not depend on the $X_1$-norm (but only on the $X_0$-norm). Hence, by following similar arguments, we prove Theorem \ref{thm2.1}. As a result, the Helly's selection theorem guarantees that there exists a subsequence $(f_i^{n_k})$ of $(f_i^n)_{n \ge i}$ converging pointwise to a $BV$ function $(f_i)$ in $([0,T])$ for each $i \ge 1$ and $T \in (0, +\infty)$ i.e.
\begin{eqnarray}\label{conv2.2}
	f_i^{n_k}(t) \to f_i(t) \ \ \ \mbox{as} \ k \to \infty, \ \forall t\in [0,T], \forall i\ge 1.
\end{eqnarray}

Let us consider $t \in(0, +\infty)$ and $r \ge i$. Now, for $n_k \ge r$, it can be deduced from \eqref{eqnum} that
\begin{align*}
\sum_{i=1}^{r}f_i^{n_k}(t) \le \|f^{\text{in}}\|_0,
\end{align*} 
and \eqref{lt2} implies
\begin{align*}
\sum_{i=1}^{r} i f_i^{n_k}(t)\le \frac{2}{R} \|f_i^{\text{in}}\|_0^{1/2} t^{-1/2}.
\end{align*}

Let us apply $k \to +\infty$ in the preceding two inequalities. Then, with the assistance of \eqref{conv2.2}, we obtain
\begin{align*}
\sum_{i=1}^{r}f_i(t) \le \|f^{\text{in}}\|_0, \ \
\text{and} \ \ \
\sum_{i=1}^{r} i f_i(t)\le \frac{2}{R} \|f_i^{\text{in}}\|_0^{1/2} t^{-1/2}.
\end{align*}

 Since the estimates above are valid for each $r \ge 1$ and the right-hand side is independent of $r$, we get
\begin{align*}
\|f(t)\|_0=\sum_{i=1}^{\infty}f_i(t) \le \|f^{\text{in}}\|_0,
\end{align*} 
and
\begin{align}\label{mass2.2}
\|f(t)\|_1=\sum_{i=1}^{\infty} i f_i(t)\le \frac{2}{R} \|f_i^{\text{in}}\|_0^{1/2} t^{-1/2},
\end{align}
which implies \eqref{mass2}.

We use \cite[Section 5]{Laurencot0} to follow a similar procedure to establish a stronger version of  Theorem \ref{thm2.2} in which \eqref{mass2} is replaced by \eqref{mass2.2}, written as

\begin{proposition}\label{prop2.2}
Let us assume that \eqref{ker1} and \eqref{ker3} holds for coagulation kernel $\theta_{i,j}$ along with initial data $f^{\text{in}}=(f_i^{\text{in}})_{i\ge 1} \in X_0^+$. Then there exists at least one solution $f$ to \eqref{drbk}-\eqref{drbkic} on $[0, +\infty)$, such that $f(t) \in X_0^+$ for each $t\in [0, +\infty)$, and
\begin{align*}
\|f(t)\|_1=\sum_{i=1}^{\infty} i f_i(t)\le \frac{2}{R} \|f_i^{\text{in}}\|_0^{1/2} t^{-1/2}, \ \ \ \ t>0.
\end{align*}
\end{proposition}

Indeed, the presented argument provides an alternative proof demonstrating that the mass will decrease over time, even when the initial density $|f^{\text{in}}|_1$ of clusters is infinite.

\begin{remark}
This paper focuses on constructing global solutions to the discrete RBK coagulation equations \eqref{drbk}-\eqref{drbkic} with coagulation rates that satisfy \eqref{ker1} and either \eqref{ker2} or \eqref{ker3}. The primary concern then is whether these solutions are unique. According to our knowledge, this remains an unresolved issue for coagulation rates satisfying $\theta_{i,j} \ge K(ij)^{\alpha}$ for some $\alpha > \frac{1}{2}$. While for $\theta_{i,j} \le K(ij)^{\alpha}$ with $0 \le \alpha \le \frac{1}{2}$, the uniqueness is shown in \cite{dacosta1}.
\end{remark}

\section*{Future Scope}
\begin{itemize}
	\item The mass decreases with time in both the discrete and the continuous RBK
	models, so there is no mass-conserving solution. There might be a possibility
	to construct mass-conserving solutions if suitable source term is associated
	with the RBK model.
	
	\item The existence of self similar solution to the discrete RBK coagulation equation
	is shown only for the class of constant kernels. But, the results on the
	existence of self similar solution to the Smoluchowski coagulation equation are
	established for the classes of bounded and unbounded coagulation kernels.
\end{itemize}

\textbf{Acknowledgements:} The authors extends sincere thanks to the Indian Institute of Technology Roorkee, Roorkee, for providing the necessary facilities that have greatly contributed to the research endeavors.\\

\textbf{Ethical Approval: } All authors involved in the research have provided their consent for the publication of this paper.\\  
 
\textbf{Competing interests:} We declare that the authors have no conflict of interest.\\
 
\textbf{Authors' contributions:} All the authors have contributed equally.\\
 
\textbf{Funding:} Not available.\\
 
\textbf{Availability of data and materials: }Data sharing not applicable to this article as no datasets were generated or analyzed during the current study.

\end{document}